\RequirePackage{fix-cm}
\documentclass[smallextended]{svjour3}       
\smartqed  
\usepackage{float}%
\usepackage{graphicx}
\usepackage{multirow}%
\usepackage{amsmath,amssymb,amsfonts}%
\usepackage{xcolor}%
\usepackage{textcomp}%
\usepackage{manyfoot}%
\usepackage{booktabs}%
\usepackage{algorithm}%
\usepackage{algorithmicx}%
\usepackage{algpseudocode}%
\usepackage{listings}%
\usepackage{threeparttable} 
\usepackage{tablefootnote}
\usepackage{adjustbox}
\usepackage{hyperref}

\newtheorem{assumption}{Assumption}

\newcommand{\hcT}{{\widehat { \cal T}}}

\newcommand\hcF{{\widehat { \cal F}}}

\newcommand\tcF{\widetilde{\mathcal{F}}}

\newcommand\cM{{\cal M}}
\newcommand\cG{{\cal G}}

\newcommand\cT{{\cal T}}

\newcommand\cI{{\cal I}}

\newcommand\cJ{{\cal J}}

\newcommand\by{{\bar y}}
\newcommand\bz{{\bar z}}
\newcommand\bx{{\bar x}}

\newcommand\bw{{\bar w}}

\newcommand{\bX}{{\mathbb X}}

\begin{document}

\title{Peaceman-Rachford Splitting Method Converges Ergodically for Solving Convex Optimization Problems 
}

\titlerunning{Ergodical Peaceman-Rachford Splitting Method}        

\author{Kaihuang Chen         \and
        Defeng Sun \and  Yancheng Yuan \and Guojun Zhang \and Xinyuan Zhao 
}



\institute{Kaihuang Chen \at
              Department of Applied Mathematics, The Hong Kong Polytechnic University, Hung Hom, Hong Kong \\
              \email{kaihuang.chen@connect.polyu.hk}           
           \and
               Defeng Sun  \at
                Corresponding author, Department of Applied Mathematics and Research Center
for Intelligent Operations Research, The Hong Kong Polytechnic University, Hung Hom, Hong Kong\\
              \email{defeng.sun@polyu.edu.hk}
        \and      Yancheng Yuan \at
               Department of Data Science and Artificial Intelligence,  The Hong Kong Polytechnic University, Hung Hom, Hong Kong \\
              \email{yancheng.yuan@polyu.edu.hk}
        \and
            Guojun Zhang \at
              Department of Applied Mathematics, The Hong Kong Polytechnic University, Hung Hom, Hong Kong \\
              \email{guojun.zhang@connect.polyu.hk}
         \and
            Xinyuan Zhao \at
              Department of Mathematics, Beijing University of Technology, Beijing, P.R. China \\
              \email{xyzhao@bjut.edu.cn}
}

\date{Received: date / Accepted: date}

\maketitle

\begin{abstract}

In this paper, we prove that the ergodic sequence generated by the Peaceman-Rachford (PR) splitting method with semi-proximal terms converges for convex optimization problems (COPs). Numerical experiments on the linear programming benchmark dataset further demonstrate that, with a restart strategy, the ergodic sequence of the PR splitting method with semi-proximal terms consistently outperforms both the point-wise and ergodic sequences of the Douglas-Rachford (DR) splitting method. These findings indicate that the restarted ergodic PR splitting method is a more effective choice for tackling large-scale COPs compared to its DR counterparts.


\end{abstract}

\section{Introduction}

Let $\mathbb{U}$, $\mathbb{W}$, $\mathbb{X}$, $\mathbb{Y}$, and $\mathbb{Z}$ be finite-dimensional real Hilbert spaces, each equipped with an inner product $\langle \cdot, \cdot \rangle$ and its corresponding norm $\|\cdot\|$. In this paper, we focus on the ergodic convergence of the Peaceman-Rachford (PR) splitting method for solving the following convex optimization problem (COP):
\begin{equation}\label{primal}
	\begin{array}{cc}
		\displaystyle\min _{y \in \mathbb{Y} , z \in \mathbb{Z}} & f_{1}(y) + f_{2}(z) \\
		\text{s.t.} & {B}_{1}y + {B}_{2}z = c,
	\end{array}
\end{equation}
where $f_1: \mathbb{Y} \to (-\infty, +\infty]$ and $f_2: \mathbb{Z} \to (-\infty, +\infty]$ are proper closed convex functions, $B_1: \mathbb{Y} \to \mathbb{X}$ and $B_2: \mathbb{Z} \to \mathbb{X}$ are given linear operators, and $c \in \mathbb{X}$ is a given vector. Let  $\sigma > 0$ be a given positive parameter. The augmented Lagrangian function for problem \eqref{primal} is defined, for any $(y, z, x) \in \mathbb{Y} \times \mathbb{Z} \times \mathbb{X}$, as
\[
L_{\sigma}(y, z; x) := f_1(y) + f_2(z) + \langle x, B_1 y + B_2 z - c \rangle + \frac{\sigma}{2} \|B_1 y + B_2 z - c\|^2.\]
The dual of  problem \eqref{primal} is
	\begin{equation}\label{dual}
		\max _{x \in \mathbb{X}}\left\{-f_{1}^{*}(-B_{1}^* x)-f_{2}^{*}(-B_{2}^* x)-\langle c, x\rangle\right\}.
	\end{equation}  
Under mild constraint qualifications \cite[Theorem 23.8]{rockafellar1970convex}, solving the dual problem \eqref{dual} is equivalent to solving the following monotone inclusion problem (MIP):
\begin{equation}\label{model:A1+A2}
  0 \in(\cG_1+\cG_2) x,  \quad x \in \mathbb{X},
\end{equation}
where $\cG_1 (x):= \partial\left(f_1^* \circ\left(-B_1^*\right)\right)(x)+c \ \  \text {and}\ \ \cG_2(x):=\partial(f_2^* \circ(-B_2^*))(x).
$
For any given maximal monotone operator $\cG: \mathbb{X}  \rightarrow 2^{\mathbb{X}}$,  we denote its resolvent by $\cJ_{\cG}:=(\cI+\cG)^{-1}$, where $\cI$ is the identity operator. The relaxed Douglas-Rachford (DR) splitting method with the relaxation factor $\rho \in (0,2]$ \cite{lions1979splitting} for solving problem \eqref{model:A1+A2} is given by
\begin{equation}\label{alg:relaxed-DR}
 u^{k+1}=\hcF^{\rm DR}_{\rho}(u^{k}):= \Big(\cI+\rho \big(\cJ_{\sigma\cG_1}\circ(2\cJ_{\sigma\cG_2}-\cI)-\cJ_{\sigma\cG_2} \big) \Big) u^k,   \ k=0,1,\ldots, 
\end{equation}
where $ u^0 \in \mathbb{X}$ is an initial point. This is equivalent to the following iterative scheme: for any $k\geq 0$,
\begin{equation*}
\left\{\begin{array}{l}
\bx_1^{k}=\cJ_{\sigma \cG_2}( u^k), \\  \bx_2^{k}=\cJ_{\sigma \cG_1}(2 \bx_1^{k}-u^k),
 \\
u^{k+1}=u^k+\rho(\bx_2^{k}-\bx_1^{k}).
\end{array}\right.
\end{equation*}
The case $\rho = 2$ is referred to as the PR splitting method, while $\rho = 1$ corresponds to the widely studied DR splitting method.

Building on the relaxed DR splitting method, Eckstein and Bertsekas \cite{eckstein1992douglas} introduced the generalized alternating direction method of multipliers (GADMM) for solving the COP \eqref{primal}. Later, in extending the semi-proximal ADMM (sPADMM) framework with larger dual step lengths studied in \cite{fazel2013hankel} to encompass the GADMM, Xiao et al. \cite{xiao2018generalized} proposed a  preconditioned ADMM (pADMM) with semi-proximal terms (or degenerate proximal ADMM) for solving the COP \eqref{primal}, as described in Algorithm \ref{alg:pADMM}.
 \begin{algorithm}[htp] 
	\caption{A pADMM method for solving COP \eqref{primal}}
	\label{alg:pADMM}
	\begin{algorithmic}[1]
		\State {Input: Let $\cT_1$ and $\cT_2$ be two self-adjoint, positive semidefinite linear operators on $\mathbb{Y}$ and $\mathbb{Z}$, respectively. Denote $w=(y,z,x)$ and $\bw=(\by,\bz,\bx)$. Select an initial point $w^{0} = (y^{0}, z^{0}, x^{0}) \in \operatorname{dom}(f_1)\times \operatorname{dom}(f_2) \times \mathbb{X}$. Set the parameters $\sigma > 0$ and $\rho\in (0,2]$.} 
    \For{$k=0,1,...,$ \vspace{3pt}}
		\State {Step 1. $\bz^{k}=\underset{z \in \mathbb{Z}}{\arg \min }\left\{L_\sigma\left(y^k, z ; x^k\right)+\frac{1}{2}\|z-z^{k}\|_{\mathcal{T}_2}^2\right\}$;}
		\State{Step 2. $\bx^{k}={x}^k+\sigma (B_{1}{y}^{k}+B_{2}\bz^{k}-c) $;}
		\State {Step 3. $\by^{k}=\underset{y \in \mathbb{Y}}{\arg \min }\left\{L_\sigma\left(y, \bz^{k} ; \bx^{k}\right)+\frac{1}{2}\|y-y^{k}\|_{\mathcal{T}_1}^2\right \}$;} 
        \State {Step 4. $w^{k+1}= (1-\rho){w}^{k}+\rho\bw^k$;}
        \EndFor 
        \vspace{3pt} 	
	\end{algorithmic}
\end{algorithm}

When \(\cT_1 = 0\) and \(\cT_2 = 0\), the pADMM with $\rho = 2$ reduces to the PR splitting method \cite{eckstein1992douglas,xiao2018generalized}, while the pADMM with $\rho = 1$ corresponds to the DR splitting method \cite{gabay1983chapter}. For general positive semidefinite linear operators $\mathcal{T}_1$ and $\mathcal{T}_2$, we refer to the pADMM as the PR splitting method with semi-proximal terms when $\rho = 2$, or as the DR splitting method with semi-proximal terms when $\rho = 1$. These semi-proximal terms are crucial for simplifying subproblem solutions and improving the scalability of the PR and DR splitting methods, especially for large-scale COPs \eqref{primal}.

It has long been observed that the PR splitting method is typically faster than the DR splitting method whenever it converges point-wisely \cite{gabay1983chapter,lions1979splitting}. Unfortunately, since $\hcF^{\rm DR}_{2}$ defined in \eqref{alg:relaxed-DR} is merely a non-expansive operator, the PR splitting method does not necessarily converge point-wisely for general MIPs \eqref{model:A1+A2} (see e.g., \cite{monteiro2018complexity}). Consequently, previous studies have primarily focused on studying the point-wise convergence of the PR splitting method under strong monotonicity assumptions on $\cG_1$ or $\cG_2$ \cite{dong2010family,giselsson2016linear,monteiro2018complexity,sim2023convergence}. However, these strong monotonicity conditions are rarely met in practice, particularly for important problem classes like COPs \eqref{primal}, which significantly limit the practical applicability of the PR splitting method. Moreover, motivated by the impressive performance of the restarted ergodic primal-dual hybrid gradient method for linear programming (LP) \cite{applegate2021practical}, we evaluate the restarted ergodic PR splitting method with semi-proximal terms on LP instances (see Section \ref{sec:4}). Numerical results show that it significantly outperforms both the point-wise and ergodic sequences of the DR splitting method with semi-proximal terms.

 Let the sequences \(\{w^t =(y^t, z^t, x^t)\}\) and \(\{\bar{w}^t=(\bar{y}^t, \bar{z}^t, \bar{x}^t)\}\) be generated by the pADMM in Algorithm \ref{alg:pADMM}. Without imposing any strong monotonicity assumption, here we focus on   analyzing the convergence of    the  two ergodic sequences \(\{w^k_a =(y^k_a, z^k_a, x^k_a) \} \) and \(\{\bar{w}^k_a= (\bar{y}^k_a, \bar{z}^k_a, \bar{x}^k_a)\} \), where for each $k\geq 0$, 
 \[  (y^k_a, z^k_a, x^k_a): = \frac{1}{k+1} \sum_{t=0}^{k} (y^t, z^t, x^t)\ {\rm and}  \ (\bar{y}^k_a, \bar{z}^k_a, \bar{x}^k_a):=    \frac{1}{k+1} \sum_{t=0}^{k} (\bar{y}^t , \bar{z}^t , \bar{x}^t ).\]
By applying Baillon's non-linear ergodic theorem \cite{baillon1975theoreme} to the non-expansive operator \(\hcF^{\rm DR}_{2}\) in \eqref{alg:relaxed-DR}, the convergence of the ergodic sequence $\{u^k_a := \frac{1}{k+1}\sum_{t=0}^{k} u^t\}$ can be directly obtained in the PR splitting method to solve MIP \eqref{model:A1+A2}. Furthermore, the continuity of the resolvent operator \(\cJ_{\sigma\cG_2}(\cdot)\) ensures the convergence of the sequence \(\{\cJ_{\sigma\cG_2}(u^k_a)\}\). However, these results do not imply the convergence of $\{\bw^k_a\}$ or $\{w^k_a\}$ of the PR splitting method to solve COP \eqref{primal}. One key reason is that averaging and applying the resolvent operator cannot generally be interchanged. {An analytical example in Appendix \ref{sec-example} further highlights the challenges of analyzing the ergodic sequence of the PR splitting method by showing that the sequence \(\{w^k\}\) generated by the PR splitting method for solving COP \eqref{primal} may be unbounded.} To our knowledge, whether the ergodic sequence of the PR splitting method converges to solve COP \eqref{primal} remains an unsolved question.

In this paper, we address this unsolved question by proving the convergence of the ergodic sequence $\{\bw^k_a\}$ of the pADMM with $\rho \in (0,2]$, which includes the PR splitting method with semi-proximal terms as a special case. Specifically, by reformulating the pADMM as a degenerate proximal point algorithm (dPPA) with a positive semidefinite preconditioner, as proposed in \cite{bredies2022degenerate}, we establish the ergodic convergence of the pADMM by analyzing the ergodic convergence of the dPPA with a relaxation factor $\rho \in (0,2]$. Furthermore, numerical results on the LP benchmark dataset show that, with a restart strategy, the ergodic sequence of the PR splitting method with semi-proximal terms outperforms both the point-wise and ergodic sequences of the DR splitting method with semi-proximal terms. This empirical evidence highlights the practical advantages of the ergodic sequence of the PR splitting method with semi-proximal terms for solving large-scale COPs \eqref{primal}. {To provide theoretical justification for these findings, in Appendix \ref{sec-ergodic complexity}, we establish an ergodic iteration complexity of $O(1/k)$ for the PR splitting method with semi-proximal terms, evaluated in terms of the objective error, the feasibility violation, and the Karush-Kuhn-Tucker residual using the $\varepsilon$-subdifferential.}

The remainder of this paper is organized as follows: Section \ref{sec:2} establishes the ergodic convergence of the dPPA for $\rho\in(0,2]$. Section \ref{sec:3} builds on this result to prove the ergodic convergence of the PR splitting method for COPs \eqref{primal}. Section \ref{sec:4} presents numerical results on the LP benchmark dataset, highlighting the superior performance of the PR splitting method’s ergodic sequence. Finally, Section \ref{sec:5} concludes the paper.

\paragraph{Notation.}
For any convex function $f: \mathbb{X} \to (-\infty, +\infty]$, we define its effective domain as $\operatorname{dom}(f) := \{x \in \mathbb{X} : f(x) < \infty\}$, its conjugate as \(f^*(x) := \sup_{z \in \mathbb{X}} \{\langle x, z \rangle - f(z)\}, \, x \in \mathbb{X}\),  and its subdifferential \(\partial f(x)=\{v \mid f(z)\geq f(x)+\langle v,x-z \rangle, \forall z \in \mathbb{X}\}, \, x \in \mathbb{X}\). Furthermore, consider a closed convex set \(C \subseteq \mathbb{X}\). We express the Euclidean projection of \(x\) onto \(C\) as \(\Pi_C(x) := \arg \min \{\|x - z\| \mid z \in C\}\). Moreover, for a linear operator \(A: \mathbb{X} \to \mathbb{Y}\), its adjoint is denoted by \(A^*\) and $\lambda_1(AA^{*})$ represents the largest eigenvalue of $AA^{*}$. Additionally, for any self-adjoint, positive semidefinite linear operator $\mathcal{M}: \mathbb{X} \to \mathbb{X}$, we define the semi-norm as \(\|x\|_{\cM} := \sqrt{\langle x, x \rangle_{\mathcal{M}}}=\sqrt{\langle x, \mathcal{M} x \rangle}\) for any $x \in \mathbb{X}$. Finally, for a non-expansive operator \(\hcF : \bX \to \bX\), the set of fixed points of \(\hcF\) is denoted by \(\operatorname{Fix} \hcF\).

\section{Ergodic convergence of the dPPA}\label{sec:2}
Let $\cT$ be a maximal monotone operator from $\mathbb{W}$ into itself.
Consider the following inclusion problem:\begin{equation}\label{model:inclusion}
		\text{ find } w\in \mathbb{W} \text{ such that } 0\in \cT w.
	\end{equation}
 Assume that  $\mathcal{M}: \mathbb{W} \to \mathbb{W}$ is an admissible preconditioner for the maximal monotone operator $\mathcal{T}$, that is, $\mathcal{M}$ is a linear, bounded, self-adjoint and positive semidefinite operator such that $\hcT:=(\cM+\cT)^{-1}\cM$ is single-valued and has a full domain.  
	For a given  $w^0 \in \mathbb{W}$, the dPPA \cite{bredies2022degenerate} for solving {the inclusion} problem \eqref{model:inclusion} is given by:
\begin{equation}\label{alg:dPPA}
		  w^{k+1}= (1-\rho) w^k + \rho \bw^k, \ {\bw^k= {\hcT w^k= (\mathcal{M}+\mathcal{T})^{-1} \mathcal{M} w^k}, \ k=0,1,\ldots, }
	\end{equation}
	where $\rho\in (0,2]$.  When $\mathcal{M}$ is the identity mapping, the dPPA reduces to the classical PPA introduced by Rockafellar in \cite{rockafellar1976a}. By carefully selecting $\mathcal{M}$, the computation of \(\hcT\) can be greatly simplified, as discussed in \cite{bredies2022degenerate,sun2024accelerating}.

    The global point-wise convergence of the dPPA for $\rho\in(0,2)$ has been established by Bredies et al. \cite{bredies2022degenerate} under the assumption that $(\cM+\cT)^{-1}$ is $L$-Lipschitz continuous, i.e., there exists a constant $L \geq 0$ such that for all $v_1, v_2$ in the domain of $(\mathcal{M} + \mathcal{T})^{-1}$,
		\[
		\|(\cM+\cT)^{-1}v_1 - (\cM+\cT)^{-1}v_2\| \leq L \|v_1 - v_2\|.
		\]
    This Lipschitz continuity assumption is met by many splitting algorithms, including the pADMM, under mild conditions (see \cite{bredies2022degenerate,sun2024accelerating}). For $\rho = 2$, the dPPA may fail to converge point-wisely for solving the inclusion problem \eqref{model:inclusion} without imposing additional assumptions, such as the strong monotonicity of $\mathcal{T}$. Instead, we focus on analyzing the global convergence of the following ergodic sequences generated by the dPPA for $\rho = 2$, without requiring the strong monotonicity of $\mathcal{T}$:
\begin{equation}\label{def:ergodic}
w^k_a = \frac{1}{k+1} \sum_{t=0}^{k} w^t, \quad \bar{w}^k_a = \frac{1}{k+1} \sum_{t=0}^{k} \bar{w}^t, \quad \forall k \geq 0,
\end{equation}
where the sequences $\{w^t\}$ and $\{\bar{w}^t\}$ are generated by the dPPA \eqref{alg:dPPA}. Recall that, for a maximal monotone operator \(\mathcal{T}: \mathbb{W} \rightarrow 2^{\mathbb{W}}\) and \(\varepsilon \geq 0\), the \(\varepsilon\)-enlargement of \(\mathcal{T}\) at \(w\) \cite{burachik1997enlargement} is defined as
\[
\mathcal{T}^{\varepsilon}(w) = \{v \in \mathbb{W} : \langle w - w^{\prime}, v - v^{\prime} \rangle \geq -\varepsilon, \, \forall (w^{\prime}, v^{\prime}) \in \text{gph}(\mathcal{T})\},
\]
where \(\text{gph}(\mathcal{T}) := \{(w, v) \in \mathbb{W} \times \mathbb{W} \mid v \in \mathcal{T} w\}\). Using the $\varepsilon$-enlargement of $\mathcal{T}$, we can establish the following proposition on the ergodic convergence properties of the dPPA.
\begin{proposition}\label{prop: ergodic-rateofdPPA}
Let \(\mathcal{T} : \mathbb{W} \rightarrow 2^{\mathbb{W}}\) be a maximal monotone operator with \(\mathcal{T}^{-1}(0) \neq \emptyset\), and let \(\mathcal{M}\) be an admissible preconditioner. Then the ergodic sequences \(\{\bw_a^k\}\) and \(\{w_a^k\}\), generated by the dPPA \eqref{alg:dPPA} with \(\rho \in (0, 2]\), satisfy   for all \(k \geq 0\) and $w^* \in \mathcal{T}^{-1}(0)$ that 
\begin{enumerate}
    \item[(a)] \(\|\bw_a^{k} - w_a^{k}\|_{\mathcal{M}} \leq \frac{2}{\rho(k+1)} \|w^0 - w^*\|_{\mathcal{M}} ;\)
    \item[(b)] \(\mathcal{M}(w_a^k - \bw_a^k) \in \mathcal{T}^{\bar{\varepsilon}_a^k}(\bw_a^k),\) where \(\bar{\varepsilon}_a^k := \frac{1}{k+1} \sum_{t=0}^{k} \langle \bw^t - \bw_a^k, w^t - \bw^t \rangle_{\mathcal{M}}\) and
    $
    0 \leq \bar{\varepsilon}_a^k \leq \frac{1}{2\rho(k+1)} \|w^0 - w^*\|_{\mathcal{M}}^2 .
    $
\end{enumerate}
\end{proposition}

\begin{proof}
Note that $\hcT$ is $\cM$-firmly non-expansive, as stated in \cite[Proposition 2.3]{sun2024accelerating}, meaning that
$$
\|\widehat{\mathcal{T}} w-\widehat{\mathcal{T}} w^{\prime}\|_{\mathcal{M}}^2+\|(\cI-\hcT) w-(\cI-\hcT) w^{\prime}\|_{\mathcal{M}}^2 \leq\|w-w^{\prime}\|_{\mathcal{M}}^2, \, \forall  w, w^{\prime} \in \mathbb{W}.
$$
From this property, we derive the following key inequality for the dPPA \eqref{alg:dPPA}:
\begin{equation}\label{eq:key-inequ}
\|w^{k+1}-w^*\|_{\cM}^2  \leq\|w^k-w^*\|_{\cM}^2-\rho(2-\rho)\|w^k-\bw^k\|_{\cM}^2, \ \forall k\geq 0,  w^* \in \cT^{-1}(0).
\end{equation}
Using this inequality and the scheme \eqref{alg:dPPA}, we can complete the proof of statement (a) by noticing that for any $k\geq0$ and $w^* \in \cT^{-1}(0)$,
$$
\begin{array}{ll}
\|\bw_a^{k}-w_a^{k}\|_{\cM}& \displaystyle=\| \frac{1}{k+1}\sum_{t=0}^{k} (\bw^{t}-w^{t})\|_{\cM}=\|\frac{1}{k+1}\sum_{t=0}^{k} \frac{(w^{t+1}-w^{t})}{\rho}\|_{\cM} \\
&\displaystyle=\frac{1}{\rho(k+1)}\|w^{k+1}-w^{0}\|_{\cM} \leq\frac{2}{\rho(k+1)}\|w^{0}-w^{*}\|_{\cM}.   
\end{array}
$$
 Moreover, for any \((w^{\prime}, v^{\prime}) \in \operatorname{gph}(\mathcal{T})\), the definitions of \(w^k_{a}\) and \(\bw^{k}_{a}\) in \eqref{def:ergodic}, along with the monotonicity of $\mathcal{T}$, yield the following for any $k \geq 0$:
\begin{equation}\label{monotone-subdifferential}
\begin{array}{ll}
 &\displaystyle\langle \bw^{k}_{a}-w^{\prime},\cM(w^{k}_{a}-\bw^{k}_{a}) -v^{\prime} \rangle = \frac{1}{k+1}\sum_{t=0}^{k}\langle  \bw^{t}-w^{\prime},\cM(w^{k}_{a}-\bw^{k}_{a}) -v^{\prime}\rangle\\
   =  &\displaystyle \frac{1}{k+1}\sum_{t=0}^{k} \Big(\langle  \bw^{t}-w^{\prime},(w^{k}_{a}-\bw^{k}_{a}) - (w^{t}-\bw^{t})\rangle_{\cM}\\ & \qquad\qquad  \qquad  +\langle  \bw^{t}-w^{\prime},\cM (w^{t}-\bw^{t}) -v^{\prime}\rangle   \Big)\\
    \geq  & \displaystyle \frac{1}{k+1}\sum_{t=0}^{k} \big\langle  \bw^{t}-w^{\prime},(w^{k}_{a}-\bw^{k}_{a}) - (w^{t}-\bw^{t})\big\rangle_{\cM}\\
    = &\displaystyle \frac{1}{k+1}\sum_{t=0}^{k} \big\langle  \bw^{t}-\bw^{k}_{a} +\bw^{k}_{a}-w^{\prime},(w^{k}_{a}-\bw^{k}_{a}) -(w^{t}-\bw^{t})\big\rangle_{\cM}\\
    =  &\displaystyle \frac{1}{k+1}\sum_{t=0}^{k}\big\langle  \bw^{t}-\bw^{k}_{a},(w^{k}_{a}-\bw^{k}_{a}) - (w^{t}-\bw^{t})\big\rangle_{\cM}\\
     =  &\displaystyle - \frac{1}{k+1}\sum_{t=0}^{k} \langle  \bw^{t}-\bw^{k}_{a}, w^{t}-\bw^{t}\rangle_{\cM}=-\bar{\varepsilon}^{k}_{a}.
\end{array}   
\end{equation}
It follows from the definition of the $\varepsilon$-enlargement of $\cT$ that 
$\mathcal{M}(w_a^k - \bw_a^k) \in \mathcal{T}^{\bar{\varepsilon}_a^k}(\bw_a^k).$
Next, we prove that \(\bar{\varepsilon}^{k}_{a} \geq 0\) for all \(k \geq 0\) by contradiction. Suppose that \(\bar{\varepsilon}^{k}_{a} < 0\) for some \(k \geq 0\). Then, for any \((w^{\prime}, v^{\prime}) \in \operatorname{gph}(\mathcal{T})\), from \eqref{monotone-subdifferential}, we have
\[
\langle \bw^{k}_{a} - w^{\prime}, \mathcal{M}(w^{k}_{a} - \bw^{k}_{a}) - v^{\prime} \rangle \geq -\bar{\varepsilon}^{k}_{a}  > 0,
\]
which, combined with the maximality of \(\mathcal{T}\), implies that \((\bw^{k}_{a}, \mathcal{M}(w^{k}_{a} - \bw^{k}_{a})) \in \operatorname{gph}(\mathcal{T})\). Taking \((w^{\prime}, v^{\prime}) = (\bw^{k}_{a}, \mathcal{M}(w^{k}_{a} - \bw^{k}_{a}))\), we obtain \(0 \geq -\bar{\varepsilon}^{k}_{a}\), which contradicts the assumption $\varepsilon^k_a<0$. Thus, \(\bar{\varepsilon}^{k}_{a} \geq 0\) for all \(k \geq 0\).

Now, we establish an upper bound for \(\bar{\varepsilon}^{k}_{a}\) for any \(k \geq 0\). By the scheme of dPPA \eqref{alg:dPPA} and $\rho\in(0,2]$,  we have
\begin{equation}\label{eq:epsilon_upperbound}
    \begin{array}{ll}
\bar{\varepsilon}^{k}_{a}&\displaystyle= \frac{1}{k+1}\sum_{t=0}^{k}\langle w^{t}-\bw^{t}, \bw^{t}-\bw^{k}_{a}\rangle_{\cM}\\
&\displaystyle=\frac{1}{k+1}\sum_{t=0}^{k} \Big( \langle  \frac{1}{\rho}(w^t-w^{t+1}), w^{t} -\frac{1}{\rho}(w^t-w^{t+1})  - \bw^k_{a}  \rangle_{\cM}\Big)\\
&\displaystyle=\frac{1}{k+1}\sum_{t=0}^{k} \Big(-\frac{1}{\rho^2}\|w^t-w^{t+1}\|^{2}_{\cM}+\frac{1}{2\rho}\|w^t-w^{t+1}\|^{2}_{\cM} \\
&\displaystyle\qquad \qquad \qquad  + \frac{1}{2\rho}(\|w^{t}-\bw^{k}_{a}\|^{2}_{\cM} -\|w^{t+1}-\bw^{k}_{a}\|^{2}_{\cM})\Big)\\
&\displaystyle \leq \frac{1}{k+1}\sum_{t=0}^{k}\Big(\frac{1}{2\rho}( \|w^{t}-\bw^{k}_{a}\|^{2}_{\cM} -\|w^{t+1}-\bw^{k}_{a}\|^{2}_{\cM})\Big)\\
&\displaystyle = \frac{1}{2\rho(k+1)} \Big( \|w^{0}-\bw^{k}_{a}\|^{2}_{\cM} -\|w^{k+1}-\bw^{k}_{a}\|^{2}_{\cM}\Big)\\
&\displaystyle = \frac{1}{2\rho(k+1)} \Big( -\|w^0-w^{k+1}\|_{\cM}^2-2\langle w^{k+1}-w^{0}, w^{0}-\bw^{k}_{a} \rangle_{\cM}  \Big)\\
&\displaystyle \leq  \frac{1}{2\rho(k+1)} \Big( -\|w^0-w^{k+1}\|_{\cM}^2+2\| w^{k+1}-w^{0}\|_{\cM}\|w^{0}-\bw^{k}_{a}\|_{\cM} \Big).\\
    \end{array}
\end{equation}
Using the convexity of $\|\cdot\|_{\cM}$ and \eqref{eq:key-inequ}, we obtain that
$$
\displaystyle\|w^{0}-\bw^{k}_{a}\|_{\cM}=\|\frac{1}{k+1}\sum_{t=0}^{k}(w^{0}-\bw^{t})\|_{\cM}\leq \frac{1}{k+1}\sum_{t=0}^{k}\|w^{0}-\bw^{t}\|_{\cM}\leq 2\|w^{0}-w^{*}\|_{\cM}.
$$   
Combining this with \eqref{eq:epsilon_upperbound}, we derive   for any \(k \geq 0\) that
\begin{equation*}
    \begin{array}{ll}
\bar{\varepsilon}^{k}_{a}&\displaystyle\leq  \frac{1}{2\rho(k+1)} \Big( -\|w^0-w^{k+1}\|_{\cM}^2+2\| w^{k+1}-w^{0}\|_{\cM}\|w^{0}-w^{*}\|_{\cM} \Big)\\
&\displaystyle\leq  \frac{1}{2\rho(k+1)} \|w^{0}-w^{*}\|_{\cM}^2.
    \end{array}
\end{equation*}
This completes the proof. 
\end{proof}
Based on Proposition \ref{prop: ergodic-rateofdPPA}, we can establish the convergence of ergodic sequence \(\{\bw^k_a\}\) generated by the dPPA \eqref{alg:dPPA} in Theorem \ref{Th:convergence-ergodic-dPPA}.

\begin{theorem}\label{Th:convergence-ergodic-dPPA}
Let \(\mathcal{T} : \mathbb{W} \rightarrow 2^{\mathbb{W}}\) be a maximal monotone operator with \(\mathcal{T}^{-1}(0) \neq \emptyset\), and let \(\mathcal{M}\) be an admissible preconditioner such that \((\mathcal{M} + \mathcal{T})^{-1}\) is \(L\)-Lipschitz continuous. Then, the ergodic sequence \(\{\bw^k_a\}\) generated by the dPPA \eqref{alg:dPPA} with \(\rho \in (0, 2]\) converges to a point in \(\mathcal{T}^{-1}(0)\).
\end{theorem}
\begin{proof}
Suppose that \(\mathcal{M} = \mathcal{C} \mathcal{C}^{*}\) is a decomposition of \(\mathcal{M}\) according to \cite[Proposition 2.3]{bredies2022degenerate}, where \(\mathcal{C} : \mathbb{U} \rightarrow \mathbb{W}\) is an injective operator. Since \((\mathcal{M} + \mathcal{T})^{-1}\) is \(L\)-Lipschitz continuous and \(\|\mathcal{C}^{*} w\| = \|w\|_{\mathcal{M}}\) for every \(w \in \mathbb{W}\), we derive  for all \(w^{\prime} \in \mathbb{W}\) and \(w^{*} \in \mathcal{T}^{-1}(0)\) that
\[
\|\widehat{\mathcal{T}} w^{\prime} - \widehat{\mathcal{T}} w^{*}\| = \|(\mathcal{M} + \mathcal{T})^{-1} \mathcal{C C}^{*} w^{\prime} - (\mathcal{M} + \mathcal{T})^{-1} \mathcal{C C}^{*} w^{*}\| \leq L \|\mathcal{C}\| \|w^{\prime} - w^{*}\|_{\mathcal{M}}.
\]
Combining this with the $\cM$-firm non-expansiveness of $\hcT$ \cite[Proposition 2.3]{sun2024accelerating}, we conclude that
\[
\|\bw^k - w^{*}\| = \|\widehat{\mathcal{T}} w^k - w^{*}\| \leq L \|\mathcal{C}\| \|w^k - w^{*}\|_{\mathcal{M}} \leq L \|\mathcal{C}\| \|w^0 - w^{*}\|_{\mathcal{M}}.
\]
Thus, both sequences \(\{\bw^k\}\) and \(\{\bw^k_a\}\) are bounded. Furthermore, according to Proposition \ref{prop: ergodic-rateofdPPA} and the maximality of \(\mathcal{T}\), any cluster point of \(\{\bw^{k}_{a}\}\) belongs to \(\mathcal{T}^{-1}(0)\). 

To establish the uniqueness of cluster points, we define two shadow sequences as follows:
\begin{equation}\label{def:u}
  u^{k} = \mathcal{C}^{*} w^k \quad \text{and} \quad u^{k}_{a} = \frac{1}{k+1}\sum_{t=0}^{k} u^t, \quad \forall k \geq 0.  
\end{equation}
A straightforward calculation shows that
$u^{k+1} = \tcF_{\rho} u^{k}, \, \forall k \geq 0,$
where \(\tcF_{\rho} := (1 - \rho)\mathcal{I} + \rho (\mathcal{C}^{*}(\mathcal{M} + \mathcal{T})^{-1} \mathcal{C})\) with \(\rho \in (0, 2]\) is a non-expansive operator, as shown in \cite[Proposition 2.5]{sun2024accelerating}. By Baillon’s non-linear ergodic theorem \cite{baillon1975theoreme}, the sequence \(\{u^{k}_{a}\}\) converges to a point in \(\operatorname{Fix}(\tcF_{\rho})\). Using the equivalence between \(\operatorname{Fix}(\tcF_{\rho})\) and \(\mathcal{C}^{*} \mathcal{T}^{-1}(0)\) as stated in \cite[Proposition 2.5]{sun2024accelerating}, we conclude that there exists \(w^{*}_{a} \in \mathcal{T}^{-1}(0)\) such that
\[
\|u^{k}_{a} - \mathcal{C}^{*} w^{*}_{a}\| \to 0.
\]
Therefore, by the definition of \(\{u^{k}_{a}\}\) in \eqref{def:u}, we have
\[
\|w^{k}_{a} - w^{*}_{a}\|_{\mathcal{M}} = \|\frac{1}{k+1} \sum_{t=0}^{k} \mathcal{C}^{*} w^t - \mathcal{C}^{*} w^{*}_{a}\| = \|u^{k}_{a} - \mathcal{C}^{*} w^{*}_{a}\| \to 0,
\]
which, together with part (a) of Proposition \ref{prop: ergodic-rateofdPPA}, implies that
\begin{equation}\label{eq:avg-bw-w}
\|\bw^k_{a} - w^{*}_{a}\|_{\mathcal{M}}^{2} = \|\bw^k_{a} - w^{k}_{a}\|_{\mathcal{M}}^{2} + \|w^{k}_{a} - w^{*}_{a}\|_{\mathcal{M}}^{2} + 2 \langle \bw^k_{a} - w^{k}_{a}, w^{k} - w^{*}_{a} \rangle_{\mathcal{M}} \to 0.
\end{equation}
Since the sequence \(\{\bw^{k}_{a}\}\) is bounded, it must have at least one cluster point. Assume that there is a subsequence \(\{\bw^{k_i}_{a}\}\) converging to \(w^{*}\). Suppose that \(\|w^{*} - w^{*}_{a}\|_{\mathcal{M}} > 0\). By an Opial-type argument \cite[Lemma 1]{opial1967weak}, we have
\[
\underset{i \to \infty}{\liminf} \|\bw^{k_i}_{a} - w^{*}\|_{\mathcal{M}} < \liminf_{i \to \infty} \|\bw^{k_i}_{a} - w^{*}_{a}\|_{\mathcal{M}},
\]
which, combined with \eqref{eq:avg-bw-w}, implies \(\underset{i \to \infty}{\liminf} \|\bw^{k_i}_{a} - w^{*}\|_{\mathcal{M}} < 0\), a contradiction to the  positive semidefiniteness of \(\mathcal{M}\). Hence, \(\|w^{*} - w^{*}_{a}\|_{\mathcal{M}} = 0\). It follows that
\[
w^{*} = (\mathcal{M} + \mathcal{T})^{-1} \mathcal{M} w^{*} = (\mathcal{M} + \mathcal{T})^{-1} \mathcal{M} w^{*}_{a} = w^{*}_{a}.
\]
Taking any other cluster point \(w^{**}\), we can similarly show that \(w^{**} = w^{*}_{a}\). Hence, the cluster point is unique, and the sequence \(\{\bw^{k}_{a}\}\) converges to \(w^{*}_{a}\).
\end{proof}

\section{Ergodic convergence of the PR splitting method}\label{sec:3}
The Karush–Kuhn–Tucker (KKT) system of COP \eqref{primal} is given by:
\begin{equation}\label{eq:KKT}
	\quad -B_1^* {x}^* \in \partial f_1(y^*), \quad -B_2^* {x}^* \in \partial f_2(z^*), \quad B_1 y^* + B_2 z^* - c = 0.
\end{equation}
As shown in \cite[Corollary 28.3.1]{rockafellar1970convex}, $({y}^*, {z}^*) \in \mathbb{Y} \times \mathbb{Z}$ is an optimal solution to problem \eqref{primal} if and only if there exists ${x}^* \in \mathbb{X}$ such that $({y}^*, {z}^*, {x}^*)$ satisfies the KKT system. To analyze the ergodic convergence of the pADMM including the PR splitting method, we make the following assumption:
\begin{assumption}\label{ass: CQ}
The KKT system \eqref{eq:KKT} has a nonempty solution set.
\end{assumption}
Under Assumption \ref{ass: CQ}, solving the COP \eqref{primal} is equivalent to finding $w\in \mathbb{W}=\mathbb{Y}\times \mathbb{Z}\times\mathbb{X}$ such that $0 \in \cT w$, where the maximal monotone operator $\cT$ is
\begin{equation}\label{def:T}
	\cT w = \left(\begin{array}{c}
		\partial f_1(y) + B_1^* x \\
		\partial f_2(z) + B_2^* x \\
		c - B_1 y - B_2 z
	\end{array} \right), \quad \forall w = (y, z, x) \in \mathbb{W}.
\end{equation}
Additionally, since \(f_1\) and \(f_2\) are proper closed convex functions, there exist two self-adjoint and  positive semidefinite operators \(\Sigma_{f_1}\) and \(\Sigma_{f_2}\) such that:
\[
\begin{array}{ll}
&f_1(y) \geq f_1(\hat{y}) + \langle \hat{\phi}, y - \hat{y} \rangle + \frac{1}{2} \|y - \hat{y}\|_{\Sigma_{f_1}}^2, \quad \forall y, \hat{y} \in \operatorname{dom}(f_1), \hat{\phi} \in \partial f_1(\hat{y}),\\
&f_2(z) \geq f_2(\hat{z}) + \langle \hat{\varphi}, z - \hat{z} \rangle + \frac{1}{2} \|z - \hat{z}\|_{\Sigma_{f_2}}^2, \quad \forall z, \hat{z} \in \operatorname{dom}(f_2), \hat{\varphi} \in \partial f_2(\hat{z}).
\end{array}
\]
We make the following assumption to ensure that each step of the pADMM is well defined.
\begin{assumption}\label{ass: Assump-solvability}
\(\Sigma_{f_1} + B_1^* B_1 + \mathcal{T}_1\) and \(\Sigma_{f_2} + B_2^* B_2 + \mathcal{T}_2\) are positive definite.
\end{assumption}
By defining the self-adjoint linear operator \(\mathcal{M}: \mathbb{W} \rightarrow \mathbb{W}\) as
\begin{equation}\label{def:M}
	\mathcal{M} = \left[\begin{array}{ccc}
		\sigma B_1^* B_1 + \mathcal{T}_1 & 0 & B_1^* \\
		0 & \mathcal{T}_2 & 0 \\
		B_1 & 0 & \sigma^{-1} \mathcal{I}
	\end{array}\right],
\end{equation}
Sun et al. \cite{sun2024accelerating} demonstrated the following equivalence between the pADMM in Algorithm \ref{alg:pADMM} and the dPPA in \eqref{alg:dPPA}.
\begin{proposition}[\cite{sun2024accelerating}]\label{prop:equ-PADMM-dPPA}
Suppose that Assumption \ref{ass: Assump-solvability} holds. Consider {the operators} $\cT$ defined in \eqref{def:T} and $\cM$ defined in \eqref{def:M}, respectively. Then the sequence $\{w^k\}$ generated by the pADMM in Algorithm \ref{alg:pADMM} coincides with the sequence $\{w^k\}$ generated by the dPPA in \eqref{alg:dPPA} with the same initial point $ w^0 \in \mathbb{W}$. Additionally, $\cM$ is an admissible preconditioner such that $(\cM + \cT )^{-1}$ is Lipschitz continuous.
\end{proposition}

The equivalence in Proposition \ref{prop:equ-PADMM-dPPA} allows us to use the ergodic convergence results of the dPPA in Theorem \ref{Th:convergence-ergodic-dPPA} to establish the ergodic convergence of the pADMM, including the PR splitting method ($\rho=2$) with semi-proximal terms, for solving COP \eqref{primal}, as shown in the following corollary.

\begin{corollary}\label{coro:convergence-PR}
Suppose that Assumptions \ref{ass: CQ} and \ref{ass: Assump-solvability} hold. Then the ergodic sequence \(\{\bw^k_a\} = \{(\by^k_a, \bz^k_a, \bx^k_a)\}\), generated by the pADMM with $\rho\in (0,2]$ in Algorithm \ref{alg:pADMM}, converges to the point \(w^* = (y^*, z^*, x^*)\), where \((y^*, z^*)\) is a solution to problem \eqref{primal}, and \(x^*\) is a solution to problem \eqref{dual}.
\end{corollary}

{
\begin{remark}
  The example in Appendix \ref{sec-example} demonstrates that the ergodic sequence \(\{w^k_a\}\) of the PR splitting method may fail to converge for general COPs \eqref{primal}. In this context, the ergodic convergence result established in Corollary \ref{coro:convergence-PR} represents the best achievable.
\end{remark}
}

\section{Numerical experiment}
\label{sec:4}
In this section, we use the following LP as an example to evaluate the performance of the ergodic sequence of the PR splitting method with semi-proximal terms:
\begin{equation}\label{model:primalLP}
\min \{ \langle c, x \rangle \mid A_1 x = b_1, \,    A_2 x \geq b_2, \,  x \in C\},
\end{equation}
where \(A_1 \in \mathbb{R}^{m_1 \times n}\), \(A_2 \in \mathbb{R}^{m_2 \times n}\), \(b_1 \in \mathbb{R}^{m_1}\), \(b_2 \in \mathbb{R}^{m_2}\), and \(c \in \mathbb{R}^n\). The set \(C := \{x \in \mathbb{R}^n \mid l \leq x \leq u\}\), with \(l \in (\mathbb{R} \cup \{-\infty\})^n\) and \(u \in (\mathbb{R} \cup \{+\infty\})^n\). Let \(A = [A_1; A_2] \in \mathbb{R}^{m \times n}\) with \(m = m_1 + m_2\) and \(b = [b_1; b_2] \in \mathbb{R}^m\). We assume that $A$ is a non-zero matrix. The dual of problem \eqref{model:primalLP} can be expressed as:
\begin{equation}\label{model:dualLP}
\min \{ -\langle b, y \rangle +  \delta_{D}(y)+
  \delta_{C}^{*}(-z) \mid  A^{*} y + z = c,\, y\in \mathbb{R}^{m},\, z\in\mathbb{R}^{n}\},
\end{equation}
where $\delta_{D}(\cdot)$ is the indicator function over \(D:= \{ y = (y_1, y_2) \in \mathbb{R}^{m_1} \times \mathbb{R}^{m_2}_{+}\}\).

We apply the pADMM from Algorithm \ref{alg:pADMM} with \(\mathcal{T}_1 = \sigma(\lambda_1(AA^*)I_m - AA^*)\) and \(\mathcal{T}_2 = 0\) to solve problem \eqref{model:dualLP}, where \(I_m\) is the identity matrix in \(\mathbb{R}^m\). In this experiment, the pADMM with $\rho=2$ and $\rho=1$ are denoted as “PR” and “DR,” respectively, with their ergodic sequences referred to as “EPR” and “EDR.” To enhance the performance of the ergodic sequences, we apply the restart strategy from \cite{applegate2021practical,chen2024hpr}, using the following merit function based on primal and dual infeasibility:
\[
\widetilde{R}_k = \sqrt{\sigma^{-1} \|\Pi_D(b - A \bar{x}_a^k)\|^2 + \sigma\|c - A^* \bar{y}_a^k - \bar{z}_a^k\|^2}.
\]
The restarted variants are referred to as ``rEPR" and ``rEDR". We implement all tested algorithms in Julia and run them on an NVIDIA A100-SXM4-80GB GPU with CUDA 12.3. The tested algorithms terminate when the following ``optimality" measure falls below a specified tolerance \(\epsilon > 0\):
\[ \max \left\{ 
\frac{| \langle b, y \rangle - \delta^*_C(-z) - \langle c, x \rangle |}{1 + |\langle b, y \rangle - \delta^*_C(-z)| + |\langle c, x \rangle|},  
\frac{\|\Pi_D(b - Ax)\|}{1 + \|b\|},  
\frac{\|c - A^* y - z\|}{1 + \|c\|} 
\right\} \leq \epsilon.
\]
Figure \ref{fig:ex10_sub3} illustrates the performance comparison of tested algorithms with \(\sigma = 1\) on the ``ex10" instance from Mittelmann’s LP benchmark. In particular, the left subfigure indicates that PR does not necessarily converge, while the middle subfigure shows that EPR does converge. Furthermore, to reach a solution with a tolerance of \(10^{-2.5}\), EPR requires roughly half of the iterations of EDR, {which is consistent with the ratio of the upper bounds for the ergodic iteration complexity results of the PR and DR splitting methods, as shown in \eqref{Th:complexity-subdiff} and \eqref{Th:complexity-obj} in Appendix \ref{sec-ergodic complexity}.} More importantly, with the restart strategy, rEPR significantly outperforms both DR and EDR.

\begin{figure}[htp]
    \centering
    \includegraphics[width=0.9\linewidth]{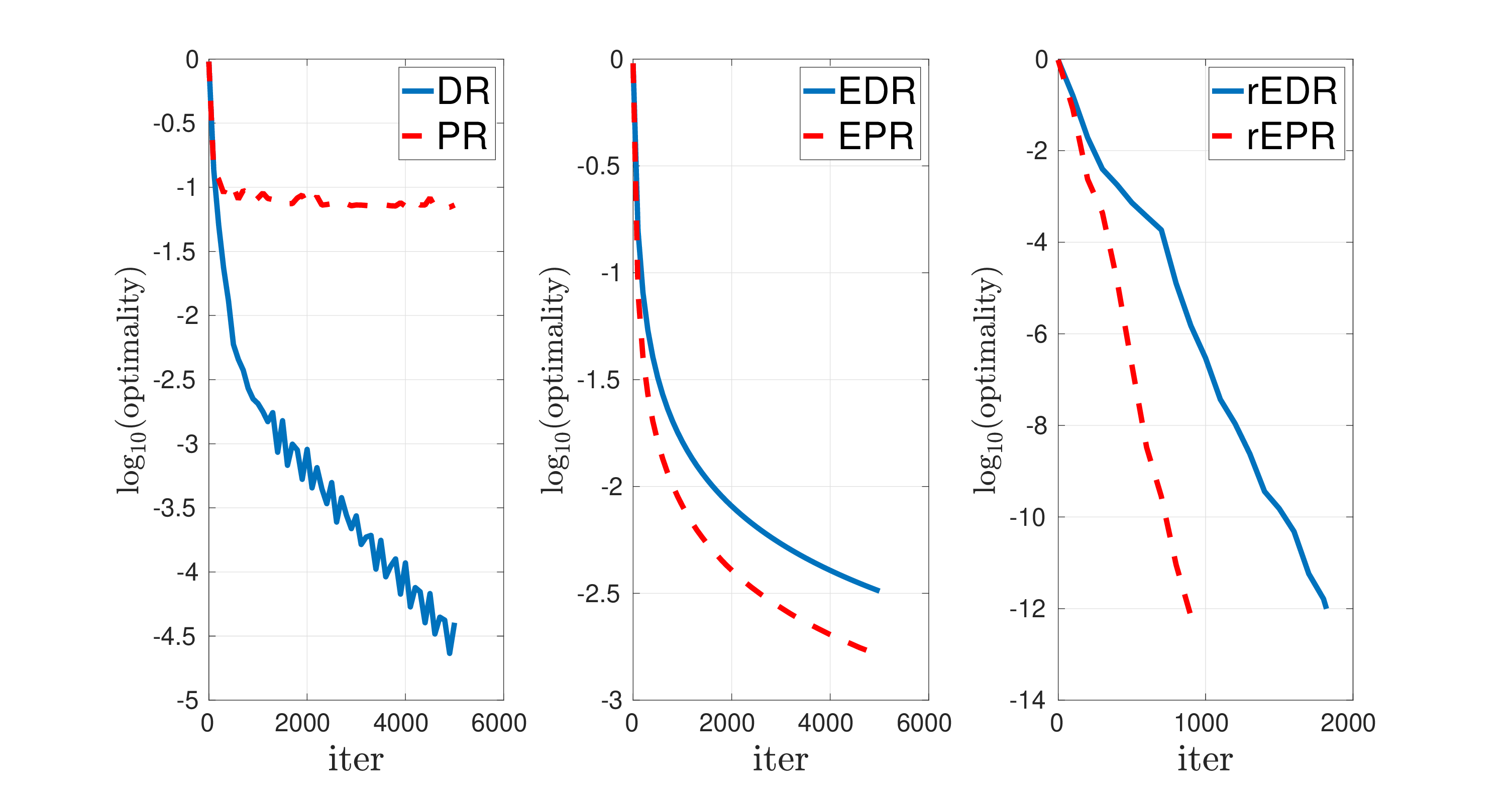}
    \caption{Performance comparison of algorithms ($\sigma=1$) on the “ex10” instance from Mittelmann’s LP benchmark}
    \label{fig:ex10_sub3}
\end{figure}

To further evaluate rEPR, rEDR, and DR, we test them on 49 instances from Mittelmann’s LP benchmark set. Each algorithm runs with a tolerance of \(\epsilon = 10^{-8}\) and a 3,600-second time limit. Figure \ref{fig:Hans_solved_time_sigmaFixed} shows performance profiles for solving times. We observe that rEPR is the fastest solver for 50\% of the problems, solves 60\% of them using half of the time required by rEDR, and solves 20\% more problems than DR. In summary, with the restart strategy, the ergodic sequence of the PR splitting method with semi-proximal terms outperforms both the point-wise and ergodic sequences of the DR splitting method with semi-proximal terms.
\begin{figure}[htp]
    \centering\includegraphics[width=0.5\linewidth]{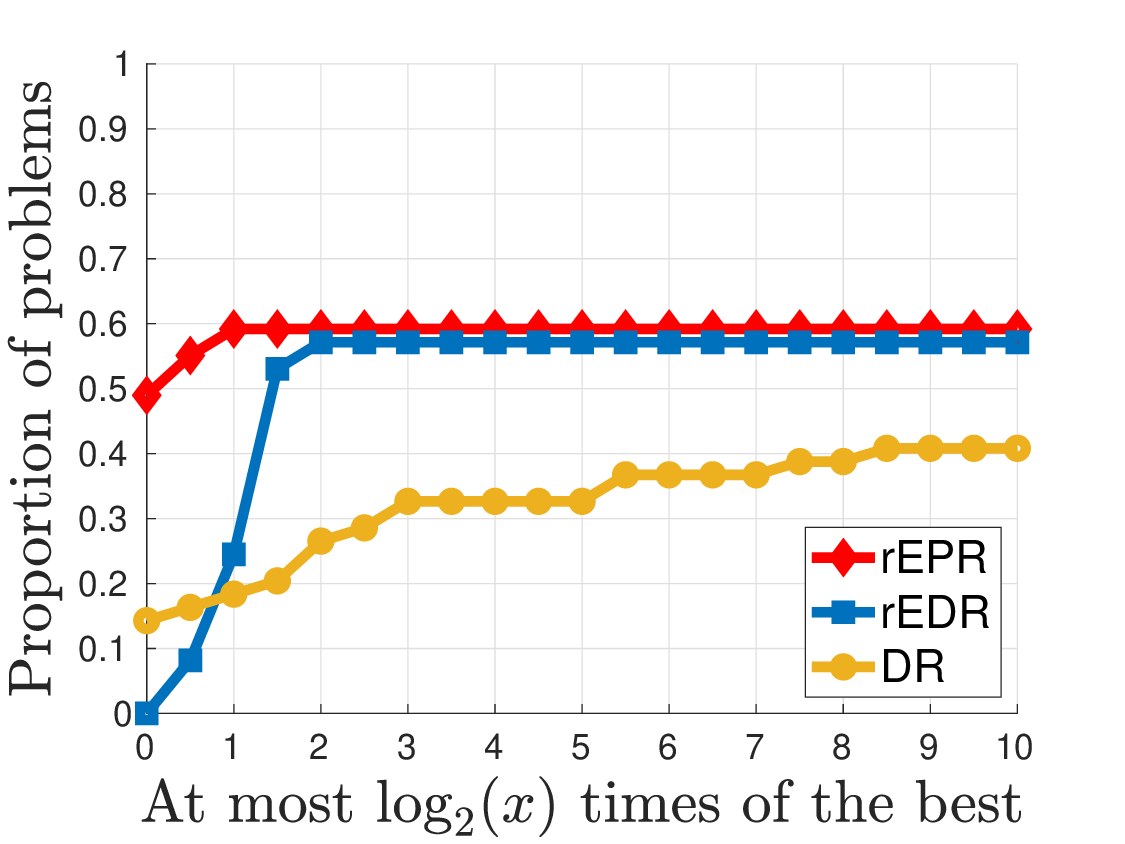}
    \caption{Performance profiles of solving times for tested algorithms ($\sigma=1$) on 49 instances from Mittelmann’s LP benchmark set}
    \label{fig:Hans_solved_time_sigmaFixed}
\end{figure}
\section{Conclusion} \label{sec:5}
To establish the ergodic convergence of the PR splitting method for solving the COP \eqref{primal}, we first proved the ergodic convergence of the dPPA with $\rho \in (0, 2]$. Leveraging the equivalence between the pADMM and the dPPA, we then demonstrated the ergodic convergence of the pADMM with $\rho \in (0, 2]$, including the PR splitting method with semi-proximal terms. Numerical results on the LP benchmark dataset revealed that, with a restart strategy, the ergodic sequence of the PR splitting method with semi-proximal terms consistently outperforms both the point-wise and ergodic sequences of the DR splitting method with semi-proximal terms. These results suggested that the restarted ergodic PR splitting method might be a more effective approach for solving large-scale COPs compared to its DR counterparts. For future research, it would be interesting to investigate why the restart strategy significantly enhances the performance of the PR splitting method’s ergodic sequence.


\section*{Statements and Declarations}
\noindent\textbf{Funding} {\thanks{The work of Defeng Sun was supported by the Research Center for Intelligent Operations Research, RGC Senior Research Fellow Scheme No. SRFS2223-5S02, and  GRF Project No. 15304721. The work of Yancheng Yuan was supported by Early Career Scheme Project No. 25305424 and the Research Center for Intelligent Operations Research. The work of Xinyuan Zhao was supported in part by the National Natural Science Foundation of China under Project No. 12271015.}}

\noindent\textbf{Conflict of interest}   The authors declare that they have no conflict of interest.

%
%

\bibliographystyle{spmpsci}      
\bibliography{ePR_ref}   

\begin{thebibliography}{10}
\providecommand{\url}[1]{{#1}}
\providecommand{\urlprefix}{URL }
\expandafter\ifx\csname urlstyle\endcsname\relax
  \providecommand{\doi}[1]{DOI~\discretionary{}{}{}#1}\else
  \providecommand{\doi}{DOI~\discretionary{}{}{}\begingroup
  \urlstyle{rm}\Url}\fi

\bibitem{adona2019iteration}
Adona, V.A., Gon{\c{c}}alves, M.L., Melo, J.G.: Iteration-complexity analysis
  of a generalized alternating direction method of multipliers.
\newblock J. Global Optim. \textbf{73}, 331--348 (2019)

\bibitem{applegate2021practical}
Applegate, D., D{\'\i}az, M., Hinder, O., Lu, H., Lubin, M., O'Donoghue, B.,
  Schudy, W.: Practical large-scale linear programming using primal-dual hybrid
  gradient.
\newblock In: Advances in Neural Information Processing System, vol.~34, pp.
  20243--20257 (2021)

\bibitem{baillon1975theoreme}
Baillon, J.B.: Un th{\'e}or{\`e}me de type ergodique pour les contractions non
  lin{\'e}aires dans un espace de hilbert.
\newblock CR Acad. Sci. Paris S{\'e}r. AB \textbf{280}, 1511--1514 (1975)

\bibitem{bredies2022degenerate}
Bredies, K., Chenchene, E., Lorenz, D.A., Naldi, E.: Degenerate preconditioned
  proximal point algorithms.
\newblock SIAM J. Optim. \textbf{32}(3), 2376--2401 (2022)

\bibitem{burachik1997enlargement}
Burachik, R.S., Iusem, A.N., Svaiter, B.F.: Enlargement of monotone operators
  with applications to variational inequalities.
\newblock Set-Valued Anal. \textbf{5}, 159--180 (1997)

\bibitem{chambolle2016ergodic}
Chambolle, A., Pock, T.: On the ergodic convergence rates of a first-order
  primal--dual algorithm.
\newblock Math. Program. \textbf{159}(1), 253--287 (2016)

\bibitem{chen2024hpr}
Chen, K., Sun, D., Yuan, Y., Zhang, G., Zhao, X.: {HPR-LP}: {A}n implementation
  of an {HPR} method for solving linear programming.
\newblock arXiv preprint arXiv:2408.12179  (2024)

\bibitem{cui2016convergence}
Cui, Y., Li, X., Sun, D., Toh, K.C.: On the convergence properties of a
  majorized alternating direction method of multipliers for linearly
  constrained convex optimization problems with coupled objective functions.
\newblock J. Optim. Theory Appl. \textbf{169}(3), 1013--1041 (2016)

\bibitem{davis2016convergence}
Davis, D., Yin, W.: Convergence rate analysis of several splitting schemes.
\newblock In: Splitting Methods in Communication, Imaging, Science, and
  Engineering, pp. 115--163. Springer (2016)

\bibitem{dong2010family}
Dong, Y., Fischer, A.: A family of operator splitting methods revisited.
\newblock Nonlinear Anal. \textbf{72}(11), 4307--4315 (2010)

\bibitem{eckstein1992douglas}
Eckstein, J., Bertsekas, D.P.: On the {D}ouglas—{R}achford splitting method
  and the proximal point algorithm for maximal monotone operators.
\newblock Math. Program. \textbf{55}(1), 293--318 (1992)

\bibitem{fazel2013hankel}
Fazel, M., Pong, T.K., Sun, D., Tseng, P.: Hankel matrix rank minimization with
  applications to system identification and realization.
\newblock SIAM J. Matrix Anal. Appl. \textbf{34}(3), 946--977 (2013)

\bibitem{gabay1983chapter}
Gabay, D.: Applications of the method of multipliers to variational
  inequalities.
\newblock In: M.~Fortin, R.~Glowinski (eds.) Augmented Lagrangian Methods:
  Applications to the Numerical Solution of Boundary-Value Problems,
  \emph{Studies in Mathematics and Its Applications}, vol.~15, pp. 299--331.
  North-Holland Publishing Company (1983)

\bibitem{giselsson2016linear}
Giselsson, P., Boyd, S.: Linear convergence and metric selection for
  {D}ouglas-{R}achford splitting and {ADMM}.
\newblock IEEE Trans. Automat. Control \textbf{62}(2), 532--544 (2016)

\bibitem{lions1979splitting}
Lions, P.L., Mercier, B.: Splitting algorithms for the sum of two nonlinear
  operators.
\newblock SIAM J. Numer. Anal. \textbf{16}(6), 964--979 (1979)

\bibitem{monteiro2018complexity}
Monteiro, R.D., Sim, C.K.: Complexity of the relaxed {P}eaceman--{R}achford
  splitting method for the sum of two maximal strongly monotone operators.
\newblock Comput. Optim. Appl. \textbf{70}, 763--790 (2018)

\bibitem{monteiro2013iteration}
Monteiro, R.D., Svaiter, B.F.: Iteration-complexity of block-decomposition
  algorithms and the alternating direction method of multipliers.
\newblock SIAM J. Optim. \textbf{23}(1), 475--507 (2013)

\bibitem{opial1967weak}
Opial, Z.: Weak convergence of the sequence of successive approximations for
  nonexpansive mappings  (1967)

\bibitem{rockafellar1970convex}
Rockafellar, R.T.: Convex {A}nalysis, vol.~18.
\newblock Princeton University Press (1970)

\bibitem{rockafellar1976a}
Rockafellar, R.T.: Monotone operators and the proximal point algorithm.
\newblock SIAM J. Control Optim. \textbf{14}(5), 877--898 (1976)

\bibitem{shen2016weighted}
Shen, L., Pan, S.: Weighted iteration complexity of the s{P}{A}{D}{M}{M} on the
  {K}{K}{T} residuals for convex composite optimization.
\newblock arXiv preprint arXiv:1611.03167  (2016)

\bibitem{sim2023convergence}
Sim, C.K.: Convergence rates for the relaxed {P}eaceman-{R}achford splitting
  method on a monotone inclusion problem.
\newblock J. Optim. Theory Appl. \textbf{196}(1), 298--323 (2023)

\bibitem{sun2024accelerating}
Sun, D., Yuan, Y., Zhang, G., Zhao, X.: Accelerating preconditioned
  {A}{D}{M}{M} via degenerate proximal point mappings.
\newblock arXiv preprint arXiv:2403.18618  (2024)

\bibitem{xiao2018generalized}
Xiao, Y., Chen, L., Li, D.: A generalized alternating direction method of
  multipliers with semi-proximal terms for convex composite conic programming.
\newblock Math. Program. Comput. \textbf{10}, 533--555 (2018)

\end{thebibliography}


\appendix

\section{An analytical example of the ergodic convergence of the PR splitting method}\label{sec-example}

The following example illustrates the challenge of analyzing the ergodic sequence of the PR splitting method.

\begin{example}
    \begin{equation}\label{example-1}
	\begin{array}{cc}
		\displaystyle\min _{y \in \mathbb{R} , z \in \mathbb{R}} & \quad  \ \delta_{\{0\}}(z) \\
		\text{s.t.} &\displaystyle  y + z = 0.
	\end{array}
\end{equation}
 The optimal solution to problem \eqref{example-1} and its dual problem is \((y^*, z^*, x^*) = (0, 0, 0)\). Choosing \(\mathcal{T}_1 = 0\), \(\mathcal{T}_2 = 0\), and initializing with \((y^0, z^0, x^0) = (1, 0, 1)\) and \(\sigma = 1\), we apply the pADMM in Algorithm \ref{alg:pADMM} to solve problem \eqref{example-1}. Through direct calculations, we obtain the following for any \(k \geq 0\):
\begin{equation*}
 \left\{
\begin{array}{l}
     \by^{k}=-2(1-\rho)^{k}, \\
     \bz^k=0, \\
     \bx^k=2(1-\rho)^{k}, \\
     y^k=-2k\rho(1-\rho)^{k-1}+(1-\rho)^k, \\
     z^k=0, \\
     x^k = 2k\rho(1-\rho)^{k-1}+(1-\rho)^k, \\
\end{array}
\right. 
\text{ and }
\left\{
\begin{array}{l}
\by_a^k=-\frac{2(1-(1-\rho)^{k+1})}{\rho(k+1)}, \\
\bz_a^k=0, \\
\bx_a^k=\frac{2(1-(1-\rho)^{k+1})}{\rho(k+1)}, \\
y_a^k=2(1-\rho)^k-\frac{1-(1-\rho)^{k+1}}{\rho(k+1)}, \\
z_a^k=0, \\
x_a^k=-2(1-\rho)^k+\frac{3(1-(1-\rho)^{k+1})}{\rho(k+1)},
\end{array}
\right.
\end{equation*}
where $0^0$ is defined as $1$. In particular, if $\rho = 2$, then for $k\geq0$, we have
\begin{equation*}
 \left\{
\begin{array}{l}
     \by^{k}=-2(-1)^{k}, \\
     \bz^k=0, \\
     \bx^k=2(-1)^{k}, \\
     y^k=-4k(-1)^{k-1}+(-1)^k, \\
     z^k=0, \\
     x^k = 4k(-1)^{k-1}+(-1)^k, \\
\end{array}
\right.
\text{ and }
\left\{
\begin{array}{l}
\by_a^k=-\frac{1-(-1)^{k+1}}{k+1}, \\
\bz_a^k=0, \\
\bx_a^k=\frac{1-(-1)^{k+1}}{k+1}, \\
y_a^k=2(-1)^k-\frac{1-(-1)^{k+1}}{2(k+1)}, \\
z_a^k=0, \\
x_a^k=-2(-1)^k+\frac{3(1-(-1)^{k+1})}{2(k+1)}. 
\end{array}
\right.   
\end{equation*}
It is evident that the point-wise sequence \(\{(\bx^k, \by^k)\}\) of the PR splitting method \((\rho = 2)\)  oscillates, while \(\{(x^k, y^k)\}\) diverges to infinity. In contrast, the ergodic sequence \(\{(\bx_a^k, \by_a^k)\}\) of the PR splitting method converges to the optimal solution, whereas \(\{(x_a^k, y_a^k)\}\) still diverges. Moreover, the ergodic sequence \(\{(\bx_a^k, \by_a^k)\}\) of the PR splitting method performs better than that of the DR splitting method (\(\rho = 1\)), in the sense that the odd ergodic iterations of the PR splitting method reach the solution directly: 
$$
\displaystyle  (\bx_a^k, \by_a^k)=\left\{
\begin{array}{lccccrc}
     (&-\frac{1-(-1)^{k+1}}{k+1}&,&\frac{1-(-1)^{k+1}}{k+1}&) , \quad &\rho=2,\vspace{5pt}\\
     (&-\frac{2}{k+1}&,&\frac{2}{k+1}&) ,\quad &\rho=1.
 \end{array}
 \right.
$$
\end{example}
In summary, this example illustrates that the PR splitting method may fail to achieve point-wise convergence when solving COP \eqref{primal}, while clearly highlighting the superiority of its ergodic sequences compared to those of the DR splitting method. Moreover, it reveals that the ergodic sequence $\{w^k_a\}$  of the PR splitting method does not necessarily converge for general COPs \eqref{primal}. In this sense, the ergodic convergence result of the PR splitting method established in Corollary \ref{coro:convergence-PR} represents the best achievable outcome.

\section{Ergodic iteration complexity of the PR splitting method for COPs}\label{sec-ergodic complexity}

To theoretically justify the superior performance of the ergodic PR splitting method compared to the ergodic DR splitting method, this section focuses on analyzing the ergodic iteration complexity of the pADMM framework, which includes the PR splitting method. We begin by introducing the concept of the \(\varepsilon\)-subgradient of a convex function \(f\) \cite{rockafellar1970convex}:
\begin{definition}
    Let \(f: \mathbb{X} \to (-\infty, +\infty]\) be a proper convex function, and let \(\bar{x} \in \operatorname{dom}(f)\). Given \(\varepsilon \geq 0\), the \(\varepsilon\)-subgradient of \(f\) at \(\bar{x}\) is defined as
\[
\partial_{\varepsilon} f(\bar{x}) := \left\{x^* \in \mathbb{X}^* \mid \langle x^*, x - \bar{x} \rangle \leq f(x) - f(\bar{x}) + \varepsilon, \, \forall x \in \mathbb{X}\right\}.
\]
\end{definition}
Using the optimality conditions of each subproblem in the pADMM and the concept of \(\varepsilon\)-subgradient, we derive the following lemma for the ergodic sequence \(\{\bar{w}^k_a\}\).

\begin{lemma}\label{lemma:epsilon-sub}
Suppose that Assumptions \ref{ass: CQ} and \ref{ass: Assump-solvability} hold. Let $w^*=(y^*,z^*,x^*)$ be a solution to the KKT system \eqref{eq:KKT}. Then the sequence $\{(\by^{k}_{a},\bz^{k}_{a},\bx^{k}_{a})\}$  generated by the pADMM in Algorithm \ref{alg:pADMM} with $\rho\in (0,2]$ satisfies for any $k\geq 0$, 
\begin{equation*}
 \left\{
\begin{array}{l}
\displaystyle -B_{2}^{*}\bx^{k}_{a}-\cT_2(\bz^{k}_{a}-z^{k}_{a})\in \partial_{\bar{\varepsilon}_{z}^{k}} f_{2}(\bz^{k}_a),\\
\displaystyle -B_{1}^{*}(\bx^{k}_{a}+ \sigma (B_{1}{\by}^{k}_{a}+B_{2}\bz^{k}_{a}-c)) -\cT_1(\by^{k}_{a}-y^{k}_{a})\in \partial_{\bar{\varepsilon}_{y}^{k}} f_{1}(\by^{k}_a), 
\end{array} \right.   
\end{equation*}
where 
\begin{equation}\label{def:epsilon-y-z}
\left\{
\begin{array}{ll}
\displaystyle \bar{\varepsilon}_{z}^{k}=\frac{1}{k+1}\sum_{t=0}^{k} \langle- B_{2}^{*}\bx^{t}-\cT_2(\bz^{t}-z^{t}), \bz^{t}-\bz^{k}_{a}  \rangle \geq 0,\\
\displaystyle \bar{\varepsilon}_{y}^{k}=\frac{1}{k+1}\sum_{t=0}^{k}\langle-B_{1}^{*}(\bx^{t}+ \sigma (B_{1}{\by}^{t}+B_{2}\bz^{t}-c))-\cT_1(\by^{t}-y^{t}), \by^{t}-\by^{k}_{a}  \rangle \geq 0,
\end{array} \right.
\end{equation}
and
\begin{equation}\label{epsilon-upper-bound}
\begin{array}{ll}
\displaystyle \bar{\varepsilon}_{z}^{k}+ \bar{\varepsilon}_{y}^{k}\leq \frac{1}{2\rho(k+1)} \|w^{0}-w^{*}\|_{\cM}^2.
\end{array}
\end{equation}
\end{lemma}
\begin{proof}
From the optimality conditions of the subproblems in Algorithm \ref{alg:pADMM}, we have, for any \(t \geq 0\),
\[
\left\{
\begin{array}{l}
  \displaystyle  f_2(z) \geq f_2(\bz^t) + \langle -B_2^{*} \bx^{t} - \mathcal{T}_2 (\bz^{t} - z^{t}), z - \bz^t \rangle, \quad \forall z \in \mathbb{Z}, \\
 \displaystyle   f_1(y) \geq f_1(\by^t) + \langle -B_1^{*} (\bx^t + \sigma (B_1 \by^t + B_2 \bz^t - c)) - \mathcal{T}_1 (\by^t - y^t), y - \by^t \rangle, \quad \forall y \in \mathbb{Y}.
\end{array}
\right.
\]
Summing these from \(t = 0\) to \(k\), and dividing by \(k+1\), we obtain for any $z \in \mathbb{Z}$ and $y \in \mathbb{Y}$,
\[\left\{
\begin{array}{ll}
    f_2(z) & \displaystyle \geq  \frac{1}{k+1} \sum_{t=0}^{k} \Big( f_2(\bz^t) + \langle -B_2^{*} \bx^t - \mathcal{T}_2 (\bz^t - z^t), z - \bz^t \rangle \Big), \\
    f_1(y) &\displaystyle \geq    \frac{1}{k+1} \sum_{t=0}^{k} \Big(f_1(\by^t) +\langle -B_1^{*} (\bx^t + \sigma (B_1 \by^t + B_2 \bz^t - c)) - \mathcal{T}_1 (\by^t - y^t), y - \by^t \rangle \Big).
\end{array}\right.
\]
By the convexity of $f_2(\cdot)$, we have for any $z\in \mathbb{Z}$,
\[				\begin{array}{ll}
                       f_2(z)&\displaystyle \geq f_2(\bz^k_{a})+\frac{1}{k+1}\sum_{t=0}^{k} \langle- B_{2}^{*}\bx^{t}-\cT_2(\bz^{t}-z^{t}), z-\bz^{t}  \rangle \\
                       &\displaystyle =f_2(\bz^k_{a})+\frac{1}{k+1}\sum_{t=0}^{k} \langle- B_{2}^{*}\bx^{t}-\cT_2(\bz^{t}-z^{t}), z-\bz^{k}_{a}\rangle -\bar{\varepsilon}_z^{k}, 	
				\end{array}
\]
where 
$$\bar{\varepsilon}_{z}^{k}=\frac{1}{k+1}\sum_{t=0}^{k} \langle- B_{2}^{*}\bx^{t}-\cT_2(\bz^{t}-z^{t}), \bz^{t}-\bz^{k}_{a}  \rangle$$
is non-negative by substituting $z=\bz^{k}_{a}$ in the first inequality. Hence, we have  for any $k \geq 0$,  
$$
 \displaystyle -(B_{2}^{*}\bx^{k}_{a}+\cT_2(\bz^{k}_{a}-z^{k}_{a}))=\frac{1}{k+1}\sum_{t=0}^{k} - (B_{2}^{*}\bx^{t}+\cT_2(\bz^{t}-z^{t}))\in \partial_{\bar{\varepsilon}_{z}^{k}} f_{2}(\bz^{k}_a).
$$
Similarly, for $f_1 (\cdot)$, we obtain that for all $k\geq 0$,
$$
-B_{1}^{*}(\bx^{k}_{a}+ \sigma (B_{1}{\by}^{k}_{a}+B_{2}\bz^{k}_{a}-c)) -\cT_1(\by^{k}_{a}-y^{k}_{a})\in \partial_{\bar{\varepsilon}_{y}^{k}} f_{1}(\by^{k}_a), \ 
$$
where 
$$
\displaystyle \bar{\varepsilon}_{y}^{k}=\frac{1}{k+1}\sum_{t=0}^{k}\langle-B_{1}^{*}(\bx^{t}+ \sigma (B_{1}{\by}^{t}+B_{2}\bz^{t}-c))-\cT_1(\by^{t}-y^{t}), \by^{t}-\by^{k}_{a}  \rangle \geq 0.
$$
Now, we show the upper bound of  $\bar{\varepsilon}_z^k+\bar{\varepsilon}_y^k$ for all $k\geq0$. According to the definitions of $\bar{\varepsilon}_z^k$ 
 and $\bar{\varepsilon}_y^k$ in \eqref{def:epsilon-y-z},  Step 2 of Algorithm \ref{alg:pADMM}, and the definition of $\cM$ in \eqref{def:M}, we obtain for any $k\geq 0$, 
\begin{equation*}
    \begin{array}{ll}
&\displaystyle \bar{\varepsilon}_{z}^{k}+ \bar{\varepsilon}_{y}^{k}\\
=&\displaystyle\frac{1}{k+1}\sum_{t=0}^{k} \Big( \langle- B_{2}^{*}\bx^{t}-\cT_2(\bz^{t}-z^{t}), \bz^{t}-\bz^{k}_{a} \rangle \\
 & \qquad \qquad \qquad +\langle-B_{1}^{*}(\bx^{t}+ \sigma (B_{1}{\by}^{t}+B_{2}\bz^{t}-c))-\cT_1(\by^{t}-y^{t}), \by^{t}-\by^{k}_{a}  \rangle \Big)\\
=&\displaystyle \frac{1}{k+1}\sum_{t=0}^{k} \Big( \langle \cM (w^t-\bw^t), \bw^{t} - \bw^k_{a}  \rangle \\ & \qquad \qquad \qquad -\langle B_{1}^{*}\bx^{t}, \by^{t} - \by^k_{a}  \rangle  -\langle B_{2}^{*}\bx^{t}, \bz^{t} - \bz^k_{a}  \rangle -\langle c-B_1\by^{t}-B_2\bz^{t}, \bx^{t}-\bx^{k}_a  \rangle        \Big)\\
=&\displaystyle \frac{1}{k+1}\sum_{t=0}^{k} \langle \cM (w^t-\bw^t), \bw^{t} - \bw^k_{a}  \rangle.
    \end{array}
\end{equation*}
Thus, by the definition of  $\bar{\varepsilon}_{a}^{k}$ in Proposition \ref{prop: ergodic-rateofdPPA}, and the equivalence between the pADMM and the dPPA in Proposition \ref{prop:equ-PADMM-dPPA}, we can derive  
$$\bar{\varepsilon}_{z}^{k}+ \bar{\varepsilon}_{y}^{k}=\bar{\varepsilon}_{a}^{k}\leq  \frac{1}{2\rho(k+1)} \|w^{0}-w^{*}\|_{\cM}^2.$$
This completes the proof. 
\end{proof}
To estimate the objective error, we define
\[
h(\bar{y}^k_a, \bar{z}^k_a) := f_1(\bar{y}^k_a) + f_2(\bar{z}^k_a) - f_1(y^*) - f_2(z^*), \quad \forall k \geq 0,
\]
where \((y^*, z^*)\) is a solution to the COP \eqref{primal}. Based on the ergodic properties established in Lemma \ref{lemma:epsilon-sub}, we derive the following iteration complexity with respect to the objective error, the feasibility violation, and the KKT residual based on $\varepsilon$-subdifferential in Theorem \ref{Th:complexity-pADMM-ergodic-epsilon-subdiff}.

\begin{theorem}\label{Th:complexity-pADMM-ergodic-epsilon-subdiff}
Suppose that Assumptions \ref{ass: CQ} and \ref{ass: Assump-solvability} hold. Let \(w^* = (y^*, z^*, x^*)\) be a solution to the KKT system \eqref{eq:KKT}, and define \(R_0 = \|w^0 - w^*\|_{\mathcal{M}}\). 
Then, the sequence \(\{(\bar{y}^k_a, \bar{z}^k_a, \bar{x}^k_a)\}\) generated by the pADMM in Algorithm \ref{alg:pADMM} with \(\rho \in (0,2]\) satisfies the following iteration complexity bounds for all \(k \geq 0\):
\begin{equation}\label{Th:complexity-subdiff}
\begin{array}{ll}
    &\displaystyle \displaystyle\operatorname{dist}\big(0, \partial_{\bar{\varepsilon}_y^k} f_1(\bar{y}^k_a) + B_1^* \bar{x}^k_a \big) 
    + \operatorname{dist}\big(0, \partial_{\bar{\varepsilon}_z^k} f_2(\bar{z}^k_a) + B_2^* \bar{x}^k_a\big) 
    + \|B_1 \bar{y}^k_a + B_2 \bar{z}^k_a - c\| \\
    &\displaystyle \leq \Big( \frac{\sigma \|B_1^*\| + 1}{\sqrt{\sigma}} + \|\sqrt{\mathcal{T}_2}\| + \|\sqrt{\mathcal{T}_1}\| \Big) \frac{2R_0}{\rho(k+1)},
\end{array}
\end{equation}
where \(\bar{\varepsilon}_z^k + \bar{\varepsilon}_y^k \leq \frac{1}{2\rho(k+1)} \|w^0 - w^*\|_{\mathcal{M}}^2\). Moreover, the following bound holds for the objective function:
\begin{equation}\label{Th:complexity-obj}
\begin{array}{ll}
    \big(-\frac{\|x^*\|}{\sqrt{\sigma}}\big) \frac{2R_0}{\rho(k+1)} 
    &\displaystyle \leq h(\bar{y}^k_a, \bar{z}^k_a) \\
    &\displaystyle  \leq \big(R_0 + 4 \sqrt{\sigma} \|B_1 y^*\|\big) \frac{R_0}{2\rho(k+1)} + \frac{\|x^0 + \sigma B_1 y^0\|^2}{2\rho(k+1)}.
\end{array}
\end{equation}
\end{theorem}

\begin{proof}
 According to Propositions \ref{prop: ergodic-rateofdPPA} and \ref{prop:equ-PADMM-dPPA}, we have
		$$
            \|\bw_a^{k}-w_a^{k}\|_{\cM}^{2} \leq \frac{4R_{0}^2}{\rho^2(k+1)^2}, \quad \forall k \geq 0. 
            $$
By the definition of $\cM$ in \eqref{def:M}, this can be rewritten as
			\begin{equation}\label{Th:complexity-pADMM-2-ergodic}
				\|\by^{k}_{a}-y^{k}_{a}\|^{2}_{\cT_1}+\frac{1}{\sigma}\|\sigma B_{1}(\by_a^{k}-y_a^{k})+(\bx_a^{k}-x_a^{k})\|^{2}+\|\bz_a^{k}-z_a^{k}\|_{\cT_2}^2 \leq \frac{4R_{0}^2}{\rho^2(k+1)^2}, \ \forall k\geq 0.
			\end{equation}
From Step 2 of Algorithm \ref{alg:pADMM}, we can deduce that for any $k\geq 0$,
			\begin{equation*}
				\begin{array}{ll}
					\|\sigma  B_{1}(\by^{k}_{a}-y^{k}_{a})+(\bx^{k}_{a}-x^{k}_{a})\|&=\|\sigma   B_{1}(\by^{k}_{a}-y^{k}_{a})+\sigma (B_{1}{y}^{k}_{a}+B_{2}\bz^{k}_{a}-c)\|\\
					&=\sigma\| B_{1}{\by}^{k}_{a}+B_{2}\bz^{k}_{a}-c\|,
				\end{array}
			\end{equation*}
which, together with \eqref{Th:complexity-pADMM-2-ergodic}, yields that
			\begin{equation}\label{Th:complexity-pADMM-3-ergodic}
				\| B_{1}{\by}^{k}_{a}+B_{2}\bz^{k}_{a}-c\| \leq \frac{2R_{0}}{\sqrt{\sigma}\rho(k+1)}, \quad \forall k\geq 0.
			\end{equation}
Furthermore, according to the Lemma \ref{lemma:epsilon-sub}, we have for $k\geq 0$,
\begin{equation*}
 \left\{
\begin{array}{ll}
-B_{2}^{*}\bx^{k}_{a}-\cT_2(\bz^{k}_{a}-z^{k}_{a})\in \partial_{\bar{\varepsilon}_{z}^{k}} f_{2}(\bz^{k}_a),\\
-B_{1}^{*}(\bx^{k}_{a}+ \sigma (B_{1}{\by}^{k}_{a}+B_{2}\bz^{k}_{a}-c)) -\cT_1(\by^{k}_{a}-y^{k}_{a})\in \partial_{\bar{\varepsilon}_{y}^{k}} f_{1}(\by^{k}_a), 
\end{array} \right.   
\end{equation*}
which, together with \eqref{Th:complexity-pADMM-2-ergodic} and \eqref{Th:complexity-pADMM-3-ergodic}, implies 
\begin{equation}\label{Th:complexity-pADMM-subdiff-z}
\displaystyle \operatorname{dist}\big(0,\partial_{\bar{\varepsilon}_{z}^{k}}f_2(\bz^k_a)+B_{2}^{*}\bx^{k}_a\big)\leq \|\cT_2(\bz^{k}_{a}-z^{k}_{a})\| \leq  \|\sqrt{\cT_2}\|\|\bz^{k}_{a}-z^{k}_{a}\|_{\cT_2}\leq  \|\sqrt{\cT_2}\|\frac{2R_0}{\rho(k+1)}    
\end{equation}
and 
\begin{equation}\label{Th:complexity-pADMM-subdiff-y}
\begin{array}{ll}
\operatorname{dist}\big(0,\partial_{\bar{\varepsilon}_{y}^{k}}f_1(\by^k_a)+B_{1}^{*}\bx^{k}_a\big) &\displaystyle  \leq  \sigma\|B_{1}^{*}\| \| B_{1}{\by}^{k}_{a}+B_{2}\bz^{k}_{a}-c\|+ \|\cT_1(\by^{k}_{a}-y^{k}_{a})\|  \\
&\displaystyle  \leq \Big(\frac{\sigma\|B_{1}^{*}\|}{\sqrt{\sigma}}+ \|\sqrt{\cT_1}\| \Big) \frac{2R_{0}}{\rho(k+1)}.   
\end{array}  
\end{equation}
Thus, combining \eqref{epsilon-upper-bound}, \eqref{Th:complexity-pADMM-3-ergodic}, \eqref{Th:complexity-pADMM-subdiff-z}, and \eqref{Th:complexity-pADMM-subdiff-y}, we derive the iteration complexity bound in \eqref{Th:complexity-subdiff}.

We now estimate the ergodic iteration complexity results for the objective error. From the KKT conditions in \eqref{eq:KKT}, we have, for any \(k \geq 0\),
\[
f_1(\by^k_a) - f_1(y^*) \geq \langle -B_1^* x^*, \by^k_a - y^* \rangle \text{  and  }  f_2(\bz^k_a) - f_2(z^*) \geq \langle -B_2^* x^*, \bz^k_a - z^* \rangle.
\]
Thus, it follows from \eqref{Th:complexity-pADMM-3-ergodic} that for all \(k \geq 0\),
\[
\begin{array}{ll}
    h(\by^k_a, \bz^k_a)  \geq \langle B_1 \by^k_a + B_2 \bz^k_a - c, -x^* \rangle \geq -\|x^*\| \|B_1 \by^k_a + B_2 \bz^k_a - c\|  \geq \displaystyle -\frac{2R_0 \|x^*\|}{\sqrt{\sigma} \rho(k+1)}.
\end{array}
\]
For the upper bound of the objective error, from \cite[Lemma 3.6]{sun2024accelerating}, we first have the following upper bounds:
		\begin{equation}\label{obj-upperbound}
			\begin{array}{ll}
				h(\by^k,\bz^k) & \displaystyle \leq  \langle\sigma B_1( y^*-\by^k)-\bx^k,  B_1 \by^k+B_2\bz^k-c  \rangle\\
				&\displaystyle  \quad + \langle y^*-\by^k,  \cT_1(\by^k-y^k)\rangle +\langle z^*-\bz^k , \cT_2(\bz^k-z^k)\rangle.
			\end{array}
		\end{equation}
Note that from Step 4 of Algorithm \ref{alg:pADMM} and $\rho\in (0,2]$,  we have for any $k\geq 0$,
\begin{equation}\label{Th:complexity-pADMM-4-ergodic}
			\begin{array}{ll}
			 & \displaystyle \langle y^*-\by^k,  \cT_1(\by^k-y^k)\rangle=  \big\langle y^*-(y^k+\frac{y^{k+1}-y^{k}}{\rho}),  \cT_1(\frac{y^{k+1}-y^{k}}{\rho})\big \rangle\\
    =& \displaystyle \frac{1}{2\rho}\big(\|y^k-y^*\|_{\cT_1}^2-\|y^{k+1}-y^*\|_{\cT_1}^2 \big)+\frac{\rho-2}{2\rho^2}\|y^{k+1}-y^{k}\|_{\cT_1}^2\\
    \leq& \displaystyle   \frac{1}{2\rho}\big(\|y^k-y^*\|_{\cT_1}^2-\|y^{k+1}-y^*\|_{\cT_1}^2\big). 
			\end{array}
		\end{equation}
Similarly, we also have 
   \begin{equation}\label{Th:complexity-pADMM-5-ergodic}
			\begin{array}{ll}
			 \left\langle z^*-\bz^k,  \cT_2(\bz^k-z^k)\right\rangle \leq  \frac{1}{2\rho} \big(\|z^k-z^*\|_{\cT_1}^2-\|z^{k+1}-z^*\|_{\cT_1}^2\big),\quad  \forall k\geq 0.
			\end{array}
		\end{equation}
Additionally, for ease of notation, we define
$$
\Delta_k:=x^k+\sigma B_{1}y^k,\quad  \forall k\geq 0.
$$
From Step 4 of Algorithm \ref{alg:pADMM} and $\rho\in (0,2]$, we can derive that for any $k\geq 0$,
 \begin{equation}\label{Th:complexity-pADMM-6-ergodic}
			\begin{array}{ll}
			 & \displaystyle \langle\sigma B_1( y^*-\by^k)-\bx^k,  B_1 \by^k+B_2\bz^k-c  \rangle\\
     = &\displaystyle  \langle\sigma B_1 y^*,  B_1 \by^k+B_2\bz^k-c  \rangle-\langle \bx^k+\sigma B_1\by^k,  B_1 \by^k+B_2\bz^k-c \rangle\\
     = &\displaystyle \langle\sigma B_1 y^*,  B_1 \by^k+B_2\bz^k-c  \rangle-\big\langle \Delta_{k}+\frac{\Delta_{k+1}-\Delta_{k}}{\rho}, \frac{\Delta_{k+1}-\Delta_{k}}{\rho}\big \rangle\\
     =&\displaystyle \langle\sigma B_1 y^*,  B_1 \by^k+B_2\bz^k-c \rangle-\frac{1}{2\rho}\big(\|\Delta_{k+1}\|^2-\|\Delta_{k}\|^2\big)+\frac{\rho-2}{2\rho^2}\|\Delta_{k+1}\|^2\\
     \leq &\displaystyle   \langle\sigma B_1 y^*,  B_1 \by^k+B_2\bz^k-c \rangle-\frac{1}{2\rho}\big(\|\Delta_{k+1}\|^2-\|\Delta_{k}\|^2\big).
			\end{array}
		\end{equation}
Thus, combing with \eqref{obj-upperbound}, \eqref{Th:complexity-pADMM-4-ergodic}, \eqref{Th:complexity-pADMM-5-ergodic}, and \eqref{Th:complexity-pADMM-6-ergodic}, we conclude  that for all $k\geq 0$,
			\begin{equation*}
				\begin{array}{ll}
					h(\by^k,\bz^k) & \displaystyle \leq   \frac{1}{2\rho}(\|y^k-y^*\|_{\cT_1}^2-\|y^{k+1}-y^*\|_{\cT_1}^2) +\frac{1}{2\rho}(\|z^k-z^*\|_{\cT_1}^2-\|z^{k+1}-z^*\|_{\cT_1}^2) \\
     & \displaystyle \quad + \langle\sigma B_1 y^*,  B_1 \by^k+B_2\bz^k-c  \rangle-\frac{1}{2\rho}\big(\|\Delta_{k+1}\|^2-\|\Delta_{k}\|^2\big).
				\end{array}
			\end{equation*}
It follows from the convexity of $h(\cdot)$  and \eqref{Th:complexity-pADMM-3-ergodic} that for any $k\geq 0$,
$$
\begin{array}{ll}
  &	\displaystyle h(\by^k_{a},\bz^k_{a})
    \leq    \frac{1}{k+1}\sum_{t=0}^{k} h(\by^t,\bz^t) \\
 \leq & \displaystyle \frac{1}{2\rho(k+1)}\big(\|y^{0}-y^{*}\|_{\cT_1}^2+\|z^{0}-z^{*}\|_{\cT_2}^2\big) + \big\langle\sigma B_1 y^*,  B_1 \by^k_a+B_2\bz^k_a-c \rangle+\frac{1}{2\rho(k+1)}\|\Delta_{0}\|^2\\
\leq  &\displaystyle \frac{R_{0}^2}{2\rho(k+1)}
  + \sqrt{\sigma}\| B_1 y^*\|\frac{2R_{0}}{\rho(k+1)}+\frac{1}{2\rho(k+1)}\|x^{0}+\sigma B_{1}y^{0}\|^2.
 \end{array}
$$
This completes the proof.			
\end{proof}

\begin{remark}
From the ergodic iteration complexity results in \eqref{Th:complexity-subdiff} and \eqref{Th:complexity-obj}, the ergodic sequence of the PR splitting method with semi-proximal terms has a worst-case upper bound that is half that of the DR splitting method with semi-proximal terms. This provides theoretical support for the superior performance of the ergodic PR splitting method with semi-proximal terms compared to the ergodic DR splitting method with semi-proximal terms, as demonstrated in the numerical experiments in Section \ref{sec:4}.
\end{remark}

\begin{remark}
We summarize some existing complexity results for ergodic sequences of ADMM-type algorithms closely related to our work in Table \ref{tab:complexity}. 
For more results, one can refer to \cite{adona2019iteration,chambolle2016ergodic} and the references in. 
\begin{table}[H]
\centering
\caption{\centering Ergodic iteration complexity results of ADMM-type algorithms}\label{tab:complexity}
\renewcommand{\arraystretch}{1.5} 
\adjustbox{max width=\textwidth}{%
 \begin{tabular}{cccccccc}
 \toprule
Paper & Algorithm & Proximal Operators & $\rho$ & \begin{tabular}[c]{@{}c@{}}Dual\\ Step-size\end{tabular} & \begin{tabular}[c]{@{}c@{}}Feasibility\\ Violation\end{tabular} & \begin{tabular}[c]{@{}c@{}}Objective\\ Error\end{tabular} & \begin{tabular}[c]{@{}c@{}}KKT\\ Residual\end{tabular} \\ 
\midrule
\cite{monteiro2013iteration} & ADMM & $\cT_1 = 0, \cT_2 = 0$ & 1 & 1 & $O(1/k)$ & - & $O_\varepsilon(1/k)$\tablefootnote{$O_{\varepsilon}(1/k)$ of the KKT residual: an $O(1/k)$ iteration complexity of the KKT residual based on $\varepsilon$-subdifferential in \eqref{Th:complexity-subdiff}.}\\
\cite{davis2016convergence} & GADMM & $\cT_1 = 0, \cT_2 = 0$ & $(0, 2]$ & 1 & $O(1/k)$ & $O(1/k)$ & - \\
\cite{cui2016convergence} & \begin{tabular}[c]{@{}c@{}}(Majorized)\\ sPADMM\end{tabular} & $\cT_1 \succeq 0, \cT_2 \succeq 0$ & 1 & $(0, \frac{1+\sqrt{5}}{2})$ & $O(1/k)$ & $O(1/k)$ & - \\
\cite{shen2016weighted} & sPADMM & $\cT_1 \succeq 0, \cT_2 \succeq 0$ & 1 & $(0, \frac{1+\sqrt{5}}{2})$ & $O(1/k)$ & - & $O_\varepsilon(1/k)$ \\
{Ours} & {pADMM} & {$\cT_1 \succeq 0, \cT_2 \succeq 0$} & {$(0, 2]$} & {1} & {$O(1/k)$} & {$O(1/k)$} & {$O_\varepsilon(1/k)$} \\  \bottomrule
\end{tabular}%
}

\end{table}

When \(\mathcal{T}_1 = 0\) and \(\mathcal{T}_2 = 0\), pADMM with \(\rho = 2\) reduces to the PR splitting method, which corresponds to the GADMM \cite{eckstein1992douglas} with \(\rho = 2\). According to the results of Davis and Yin \cite{davis2016convergence}, the ergodic sequence of the PR splitting method achieves an \(O(1/k)\) complexity in terms of the objective error and the feasibility violation. In contrast, we establish the ergodic iteration complexity result for the PR splitting method with semi-proximal terms with respect to the objective error, the feasibility violation, and the KKT residual based on $\varepsilon$-subdifferential. This generalization is particularly significant, as carefully chosen \(\mathcal{T}_1\) and \(\mathcal{T}_2\) can simplify the solution of subproblems in some important convex optimization problems, such as general LPs.
\end{remark}

\end{document}